\pdfoutput=1

\documentclass[
	bigMargin = false,
	font = palatino,
	final
	]{MyClass}

\usepackage{myMathHeader}
\newcommand{\ex}{\textnormal{ex}}
\renewcommand{\DiffGroup}{\mathrm{Diff}}
\renewcommand{\MetricSpace}{\mathcal{M}}
\declaretheorem[numbered=no, style=myTheorem, name={Main Theorem}]{mainThm}

\Title{Realizing the Teichmüller space \\[0.5ex] as a symplectic quotient}

\Abstract{
	Given a closed surface endowed with a volume form, we equip the space of compatible Riemannian structures with the structure of an infinite-dimensional symplectic manifold.
	We show that the natural action of the group of volume-preserving diffeomorphisms by push-forward has a group-valued momentum map that assigns to a Riemannian metric the canonical bundle.
	We then deduce that the Teichmüller space and the moduli space of Riemann surfaces can be realized as symplectic orbit reduced spaces.
}



\Institution{delft}{Delft University of Technology}

\Institution{sjtu}{Shanghai Jiao Tong University}

\Author[
 	affiliation = delft,
 	email = T.Diez@tudelft.nl,
 	]{diez}{Tobias Diez}

 \Author[
 	affiliation = sjtu,
 	email = {ratiu@sjtu.edu.cn, tudor.ratiu@epfl.ch},
 	]{ratiu}{Tudor S. Ratiu}

\begin{document}

\MakeTitle

%\tableofcontents

%\listoftodos

\section{Introduction}

In their seminal work, \textcite{AtiyahBott1983} have shown that the curvature 
yields a momentum map for the action of the group of gauge transformations on 
the space of connections on a principal bundle over a closed surface.
Thus, the moduli space of flat connections, and the related moduli space of 
central Yang--Mills connections, are realized as symplectic  reduced spaces. 
This observation has inspired a lot of research to reformulate problems in 
differential geometry in the language of infinite-dimensional symplectic 
geometry.

For example, \textcite{Donaldson2003} observed that the scalar curvature 
produces a momentum map \( \SectionMapAbb{J}_{\ex} \) for the action of the 
group \( \DiffGroup_{\mu, \ex}(M) \) of \emph{exact} volume-preserving 
diffeomorphisms of a surface \( (M, \mu) \) on the space \( \MetricSpace_\mu \) of all compatible 
Riemannian structures. Even though the symplectic quotient 
\( \SectionMapAbb{J}_{\ex}^{-1}(c) \slash \DiffGroup_{\mu, \ex}(M) \) is 
related to the Riemann moduli space, they do not coincide because the latter 
is the quotient by the full group \( \DiffGroup_{\mu}(M) \) of all 
volume-preserving diffeomorphisms. According to \textcite[p.~181]{Donaldson2003}, 
the "difficulty is that there is no way to extend the moment[um map] 
\( \SectionMapAbb{J}_{\ex} \) to an equivariant moment[um] map for the full 
action of \( \DiffGroup_{\mu}(M) \)". In this note, we will show that the 
action of \( \DiffGroup_{\mu}(M) \) admits a momentum map in a generalized sense.
This result allows us to realize the Riemann moduli space as a symplectic orbit
quotient. The main result is the following.
\begin{mainThm}
\label{thm:mainThm}
Let \( M \) be a closed surface endowed with a volume form \( \mu \).
The space \( \MetricSpace_\mu \) of Riemannian metrics on \( M \) compatible 
with \( \mu \) is endowed with the symplectic form
\begin{equation}\label{eq:applications:coadSL:symplecticFormOnMetrics}
\Omega_g (h_1, h_2) = - \frac{1}{2}\int_M
\tr\Bigl( \bigl( g^{-1} h_1\bigr) \bigl( g^{-1}\mu\bigr) 
\bigl( g^{-1}h_2\bigr) \Bigr) \, \mu \,,
\end{equation}
where \( g \in \MetricSpace_\mu \) and 
\( h_1, h_2 \in \TBundle_g \MetricSpace_\mu \).
The left action of \( \DiffGroup_\mu (M) \) on \( \MetricSpace_\mu \) 
by push-forward preserves \( \Omega \) and has a 
group-valued momentum map
\begin{equation}
\label{eq:applications:coadSL:momentumMapForMetrics}
\SectionSpaceAbb{J}: \MetricSpace_\mu \to \csCohomology^2(M, \UGroup(1)), 
\quad g \mapsto \KBundle_g M,
\end{equation}
where \( \csCohomology^2(M, \UGroup(1)) \) is the Abelian group of gauge equivalence 
classes of circle bundles with connection over \( M \)  and \( \KBundle_g M \) 
is the canonical circle bundle.
\end{mainThm}

We need to explain the terminology and notation of the theorem.
\begin{enumerate}
\item
A Riemannian metric $g$ on $M$ is \emph{compatible} with the volume form 
$\mu$ on $M$ if the volume form induced by \( g \) coincides with \( \mu \).
For such metrics, we have $\nabla \mu =0$, where $\nabla$ is the Levi-Civita 
connection defined by $g$.

The space \( \MetricSpace_\mu \) of Riemannian metrics compatible with \( \mu \) 
is identified with the space of sections of the associated bundle 
\(\FrameBundle M \times_{\SLGroup}\bigl(\SLGroup(2,\R)\slash \SOGroup(2)\bigr)\), 
where \( \FrameBundle M \) denotes the \( \SLGroup(2, \R) \)-frame bundle 
induced by the volume form \( \mu \). As such, \( \MetricSpace_\mu \) naturally 
comes with the structure of an infinite-dimensional Fr\'echet manifold.
The space of \( g \)-trace-free symmetric covariant \( 2 \)-tensors is 
the tangent space to $ \MetricSpace_\mu $ at \( g \).
In particular, in formula~\eqref{eq:applications:coadSL:symplecticFormOnMetrics}, 
$h_1$ and $h_2$ are trace-free symmetric covariant 2-tensors.
\item  Given a symmetric covariant \( 2 \)-tensor $h$ and a Riemannian metric $g$
on $M$, the notation $g^{-1} h$ is the \( (1,1) \)-tensor defined by 
$(g^{-1} h)( \alpha , X) = h(g^{-1} \alpha , X)$ for a \( 1 \)-form 
$\alpha \in \DiffFormSpace^1 (M)$ and a vector field $X \in\VectorFieldSpace(M)$, 
where $g^{-1}: \DiffFormSpace^1 (M) \to \VectorFieldSpace(M)$ is the isomorphism 
induced by the Riemannian metric. In coordinates, this amounts to the 
operation of raising the first index by the Riemannian metric $g$, 
\ie, $(g^{-1} h)^i_{\;j} = g^{ik}h_{kj}$ with $[g^{ij}]=[g_{ij}]^{-1}$.

Hence, in~\eqref{eq:applications:coadSL:symplecticFormOnMetrics}, the  
integrand is the trace of the product of three matrices, \ie, 
$\tr\Bigl( \left( g^{-1} h_1\right) \left( g^{-1}\mu\right) 
\left( g^{-1}h_2\right) \Bigr)= \tensor{(h_1)}{^i_j} \, 
\tensor{\mu}{^j_k} \, \tensor{(h_2)}{^k_i} \, $.
\item  Let \( I \) be the almost complex structure on \( M \) induced by the 
symplectic form \( \mu \) and the Riemannian metric \( g \) on $M$. The complex 
line bundle \( \ExtBundle^{1,0} M \) of holomorphic forms is called the 
\emph{canonical line bundle}. The form \( \mu \) is a non-vanishing 
section of \(\ExtBundle^{1,1}M = \ExtBundle^{1,0}M \tensorProd\ExtBundle^{0,1}M\) 
and thus induces a Hermitian metric on \( \ExtBundle^{1,0} M \). The associated 
Hermitian frame bundle \( \KBundle_g M \) is a principal circle bundle, 
called the \emph{canonical circle bundle}.
The Levi--Civita connection of \( g \) naturally induces a connection 
in \( \KBundle_g M \).
\item The set \( \csCohomology^2(M, \UGroup(1)) \) of gauge equivalence 
classes of circle bundles with connection is an Abelian group relative to
the following addition introduced in \textcite{Kobayashi1956}. Let \( P \) 
and \( \tilde{P} \) be two principal \( \UGroup(1) \)-bundles,   form their
fiber product \( P \times_M \tilde{P} \), and identify points which differ 
by the \( \UGroup(1) \)-action $(p, \tilde{p}) \cdot  z:= 
(p \cdot z, \tilde{p} \cdot z^{-1})$, for $p\in  P$, $\tilde{p} \in\tilde{P}$, 
and $z \in \UGroup(1)$. This defines a new principal 
\(\UGroup(1)\)-bundle \(P + \tilde{P} \defeq (P \times_M \tilde{P})/\UGroup(1)\),    
where the $\UGroup(1)$-action on $P + \tilde{P}$ is translation on the first 
factor. This operation is associative and commutative. The trivial bundle is 
the identity element, \ie, \( P + (M \times \UGroup(1)) \) is isomorphic to 
\( P \). Given a principal bundle \( P \), denote by \( -P \)  the 
\( \UGroup(1) \)-bundle 
having the same underlying bundle structure as \( P \) but carrying the opposite 
\( \UGroup(1) \)-action \( p \ast z \defeq p \cdot z^{-1}\), where   the 
right side is the given \( \UGroup(1) \)-action on \( P \). Then \( P + (-P) \) 
is isomorphic to the trivial bundle. Connections \( A \) on $P$ and 
\( \tilde{A} \) on $\tilde{P}$ induce the connection \(\pr_1^* A + 
\pr_2^* \tilde{A}\) on \(P \times_M \tilde{P}\), where \(\pr_1: P \times_M 
\tilde{P} \rightarrow  P\) and \(\pr_2: P \times_M \tilde{P} \to  \tilde{P}\) 
are the projections on the two factors,   which descends to a 
connection on \( P + \tilde{P} \), denoted by \( A + \tilde{A} \). The 
curvature of \( A + \tilde{A} \) is the sum of the corresponding curvatures.

Moreover, \( \csCohomology^2(M, \UGroup(1)) \) is an Abelian \emph{Fr\'echet Lie} 
group with Lie algebra \( \DiffFormSpace^1(M) \slash \dif \DiffFormSpace^0(M) \); 
see \parencite[Appendix~A]{BeckerSchenkelEtAl2014}.
\item
The concept of a group-valued momentum map is inspired by the notion of a 
momentum map in Poisson geometry as introduced by \citeauthor{LuWeinstein1990} 
\parencite{LuWeinstein1990,Lu1990} and will be discussed in detail below.
\end{enumerate}

The curvature of the canonical bundle \( \KBundle_g M \) is given by 
\( - S_g \mu \), where \( S_g \) denotes the scalar curvature of \( g \). 
Hence, symplectic reduction at the subset \( \curv^{-1}(\mu) \) of all bundles 
with constant curvature \( \mu \) yields the Riemann moduli space:
\begin{equation}
\SectionSpaceAbb{J}^{-1} (\curv^{-1}(\mu)) \slash \DiffGroup_\mu(M) =
\set[\big]{g \in \MetricSpace_\mu \given S_g = -1} \slash \DiffGroup_\mu(M).
\end{equation}
Note that, in contrast to classical symplectic reduction, we take the inverse 
image of a set and not just of a point. However, one can show that 
\( \DiffGroup_\mu(M) \) acts (infinitesimally) transitively on 
\( \curv^{-1}(\mu) \), so that the reduction is a \emph{symplectic orbit 
reduction} \parencite[Section~6.3]{OrtegaRatiu2003}.

Instead of taking the quotient with respect to \( \DiffGroup_\mu(M) \), we can 
also restrict attention to the connected component of the identity 
\( \DiffGroup_\mu(M)^\circ \). The action of \( \DiffGroup_\mu(M)^\circ \) 
is free and its momentum map is given by the same 
formula~\eqref{eq:applications:coadSL:momentumMapForMetrics}.
Thus, the symplectic quotient with respect to the 
\( \DiffGroup_\mu(M)^\circ \)-action is a smooth manifold that coincides 
with the Teichmüller space.
Moreover, the expression~\eqref{eq:applications:coadSL:symplecticFormOnMetrics} 
for the symplectic form on the space of Riemannian metrics implies that the 
reduced symplectic form is proportional to the Weil--Petersson symplectic form 
on the Teichmüller space.

\begin{remark}
In \parencite{DiezRatiu}, we consider the general setting given by a symplectic 
fiber bundle \( F \to M \). Then the space of sections of \( F \) carries a 
natural symplectic structure induced by the fiber-symplectic structure on 
\( F \) and we determine the 
group-valued momentum for the action of the automorphism group of \( F \).
From this perspective, the \hyperref[thm:mainThm]{Main Theorem} is 
deduced as a special case of the theory developed in \parencite{DiezRatiu}.
The results of that paper also imply that the symplectic geometry of 
\( \MetricSpace_\mu \) and the momentum map \( \SectionMapAbb{J} \) are 
largely determined by the finite-dimensional coadjoint orbit 
\( \SLGroup(2, \R) \slash \SOGroup(2) \). It follows that the Teichmüller 
space has two close relatives corresponding to the hyperbolic and the parabolic 
coadjoint orbits of \( \SLGroup(2, \R) \); we refer to \parencite{DiezRatiu} for 
details.
\end{remark}
\begin{remark}
Besides the scalar curvature as a geometric datum, the group-valued momentum 
map \( \SectionSpaceAbb{J} \) contains topological information in the form 
of the Chern class of \( M \). Such discrete topological data cannot be 
encoded in classical momentum maps. In particular, the momentum map 
\( \SectionMapAbb{J}_{\ex} \) for the action of the group 
\( \DiffGroup_{\mu, \ex}(M) \) of exact volume-preserving diffeomorphisms 
does not contain topological information.

On the other hand, it is a generic feature of group-valued momentum maps 
that they capture geometric as well as topological data. For example, in 
\parencite{DiezRatiu} we show that the group-valued momentum map can recover the 
Liouville class of a Lagrangian embedding or the integral helicity of fluid 
configurations as additional topological information. 
\end{remark}

\section{Group-valued momentum maps}
In order to handle the full groups of volume-preserving diffeomorphisms, we 
introduce the concept of a group-valued momentum map. 

Our starting point is the notion of a momentum map in Poisson geometry as 
introduced by \citeauthor{LuWeinstein1990} \parencite{LuWeinstein1990,Lu1990}.
Let \( (G, \varpi_G) \) be a (finite-dimensional) Poisson Lie group with dual 
group \( G^* \) and let \( (M, \varpi_M) \) be a (finite-dimensional) Poisson 
manifold. Recall that a left action of \( G \) on \( M \) is called a Poisson 
action if the action map \( G \times M \to M \) is a Poisson map, where 
\( G \times M \) is endowed with the product Poisson structure 
\( \varpi_G \times \varpi_M \). A smooth map \( J: M \to G^* \) is called a 
\emphDef{momentum map} of this action if
\begin{equation}
	\label{eq::luMomentumMap:defining}
	A^* + \varpi_M\left(\cdot, J^* A^l\right) = 0
\end{equation}
holds for all \( A \in \LieA{g} \), where \( A^* \) denotes the fundamental 
vector field on \( M \) induced by the infinitesimal action of \( A \) and 
\( A^l \in \DiffFormSpace^1(G^*) \) is the left-invariant extension of \( A \) 
seen as a functional on \( \LieA{g}^* \). If the Poisson structure 
\( \varpi_M \) is induced by a symplectic form \( \omega \) on \( M \), 
then~\eqref{eq::luMomentumMap:defining} is equivalent to
\begin{equation}
	\label{eq::luMomentumMap:symplectic}
	A^* \contr \omega + \dualPair{A}{\difLog J} = 0,
\end{equation}
where \( \difLog J \in \DiffFormSpace^1(M, \LieA{g}^*) \) is the left-logarithmic 
derivative of \( J \) defined by 
\( (\difLog J)_m (X_m) = J(m)^{-1} \ldot \tangent_m J (X_m) \).
Note that this equation does no longer use the fact that \( G \) is a Poisson 
Lie group. Indeed, this identity still makes sense if the momentum map is 
replaced by a smooth map \( J: M \to H \) with values in an arbitrary Lie 
group \( H \), as long as there is a duality between the Lie algebras of \( G \) 
and \( H \). This observation leads to our generalization of Lu's momentum map.

A \emphDef{dual pair of Lie algebras} (not necessarily finite-dimensional) 
consists of two Lie algebras \( \LieA{g} \) and \( \LieA{h} \), which are in 
duality through a given (weakly) non-degenerate bilinear map \( \kappa: \LieA{g} 
\times \LieA{h} \to \R \). Using notation stemming from functional analysis, 
we write the dual pair as \( \kappa(\LieA{g}, \LieA{h}) \).
Two Lie groups \( G \) and \( H \) are said to be \emphDef{dual} to each other 
if there exists a non-degenerate bilinear form \( \kappa: \LieA{g} \times 
\LieA{h} \to \R \) relative to which the associated Lie algebras are in duality.
We use the notation \( \kappa(G, H) \) in this case.
\begin{defn}
Let \( M \) be \(  G \)-manifold endowed with a symplectic form \( \omega \).
A \emphDef{group-valued momentum map} is a pair \( (J, \kappa) \), where 
\( \kappa(G, H) \) is a dual pair of Lie groups and \( J: M \to H \) is a 
smooth map satisfying
\begin{equation}
\label{eq::momentumMap:DefEq}
A^* \contr \omega + \kappa(A, \difLog J) = 0
\end{equation}
for all \( A \in \LieA{g} \).
\end{defn}
We emphasize that the concept of a group-valued momentum map is a vast 
generalization of many notions of momentum maps appearing in the literature 
including circle-valued, cylinder-valued, and Lie algebra-valued momentum maps.
Most of the well-known results about the classical momentum map (such as 
equivariance properties and the Bifurcation Lemma) generalize in a natural way 
to group-valued momentum maps; see \parencite{DiezRatiu,DiezThesis}.

For our purposes, the following dual pair is of particular relevance.
\begin{example}
Let \( M \) be a compact manifold endowed with a volume form \( \mu \).
The group \( G = \DiffGroup_\mu(M) \) of volume-preserving diffeomorphisms is 
a Fr\'echet Lie group with Lie algebra consisting of \( \mu \)-divergence-free 
vector fields \( X \) on \( M \):
\begin{equation}
\LieA{g} = \VectorFieldSpace_\mu(M) \isomorph \set{X \in \VectorFieldSpace(M) 
\given \dif (X \contr \mu) = 0}.
\end{equation}
Hence, \( \VectorFieldSpace_\mu(M) \) can be identified with the space 
\( \clDiffFormSpace^{\dim M-1}(M) \) of closed sub-top forms so that 
\( \LieA{h} \defeq \DiffFormSpace^1(M) \slash \dif \DiffFormSpace^0(M) \) is 
the regular dual with respect to the weakly non-degenerate integration paring
\begin{equation}
\label{eq::diffAction:dualPairOfVolPresVectorFields}
\kappa(X, \equivClass{\alpha}) \defeq \int_M (X \contr \alpha) \, \mu.
\end{equation}
We now observe that a \( 1 \)-form \( \alpha \) on \( M \) can be seen as a 
connection on the trivial principal circle bundle \( M \times \UGroup(1) \to M \). 
From this point of view, \( \LieA{h} \) parametrizes gauge equivalence classes 
of connections on the trivial circle bundle. Thus, it is natural to think of it 
as the Lie algebra of the Abelian group \(H\defeq \csCohomology^2(M,\UGroup(1))\) 
of all principal circle bundles with connections, modulo gauge equivalence.
This heuristic argument can be made rigorous using the theory of Cheeger--Simons 
differential characters \parencite{BaerBecker2013}; see \parencite{DiezRatiu}.
Summarizing, we get a dual pair \( \kappa\bigl(\DiffGroup_\mu(M), 
\csCohomology^2(M, \UGroup(1))\bigr) \) of Lie groups. In other words, a 
group-valued momentum map for an action of \( \DiffGroup_\mu(M) \) takes values 
in \( \csCohomology^2(M, \UGroup(1)) \).
\end{example}

\section{Proof of the Main theorem}
In the sequel, we will prove the \hyperref[thm:mainThm]{Main Theorem} by means 
of two lemmas, which compute the two terms in the momentum map 
relation~\eqref{eq::momentumMap:DefEq}.

\begin{lemma}
In the setting of the \hyperref[thm:mainThm]{Main Theorem}, we have
\begin{equation}
\Omega_g(X \ldot g, h) = - \int_M X^i \bigl(\tensor{\mu}{_i_k} 
\nabla_j \tensor{h}{^k^j}\bigr) \, 
\mu \, ,
\qedhere
\end{equation}
where \( X \ldot g = - \difLie_X g \) denotes the fundamental vector field 
induced by the action of the divergence-free vector field 
\( X \in \VectorFieldSpace_\mu(M) \),  evaluated at $g \in \MetricSpace_\mu$,  
and \( h \in \TBundle_g \MetricSpace_\mu \) is
a trace-free symmetric covariant \( 2 \)-tensor. 
\end{lemma}
\begin{proof}
Let \( \nabla \) be the Levi--Civita connection associated to the metric \( g \).
Since \( g \) is parallel with respect to \( \nabla \), we have
\begin{equation}
(\difLie_X g)_{ij} = \nabla_i X_j + \nabla_j X_i
\end{equation}
for every vector field \( X \).
Moreover, the tensor \( \tensor{\mu}{_i_k} \tensor{h}{^k_j} \) is symmetric 
in the indices \( i \) and \( j \), because \( h \) is trace-free and thus 
Hamiltonian with respect to the symplectic form \( \mu \).
Hence, we obtain
\begin{equation}
\begin{split}
\int_M \tr\Bigl( \left( g^{-1} (\difLie_X g) \right) \left( g^{-1}\mu \right) 
\left( g^{-1}h\right) \Bigr)   \, \mu
&= \int_M \bigl(\nabla^i X_j + \nabla_j X^i\bigr) \tensor{\mu}{^j_k} 
\tensor{h}{^k_i} \, \mu  \\
&= 2\int_M \bigl(\nabla_j X^i\bigr) \tensor{\mu}{_i_k} \tensor{h}{^k^j} \, \mu \\
&= - 2 \int_M X^i \tensor{\mu}{_i_k} \nabla_j \tensor{h}{^k^j} \, \mu \, ,
\end{split}
\end{equation}
where we integrated by parts and used  \(\nabla  \mu =0\).
\end{proof}

\begin{lemma}
In the setting of the \hyperref[thm:mainThm]{Main Theorem}, the map
\begin{equation}
\SectionSpaceAbb{J}: \MetricSpace_\mu \to \csCohomology^2(M, \UGroup(1)), 
\quad g \mapsto \KBundle_g M,
\end{equation}
has logarithmic derivative \( (\difLog \SectionSpaceAbb{J})_g: \TBundle_g 
\MetricSpace_\mu\to\DiffFormSpace^1(M) \slash \dif \DiffFormSpace^0(M)\) given by
\begin{equation}
(\difLog \SectionSpaceAbb{J})_g (h)_i = 
\tensor{\mu}{_i_k} \nabla_j \tensor{h}{^k^j} \mod 
\dif \DiffFormSpace^0(M),
\end{equation}
where \( g \in \MetricSpace_\mu \) and \( h \in \TBundle_g \MetricSpace_\mu \).
\end{lemma}
\begin{proof}
This follows from the general results of \parencite{DiezRatiu}. To keep this
paper self-contained, we give here a direct proof in our particular case 
under the additional assumption that the first singular homology group of \( M \) 
is trivial.

We need to determine the derivative of the holonomy map of \( \KBundle_g M \) 
with respect to \( g \).
For this, let \( \gamma \) be a closed loop at \( m \in M \).
As \( \sHomology_1(M, \Z) \) is trivial, there exists a smooth contraction 
\( \Sigma: [0, 1] \times [0, 1] \to M \) of \(\gamma\) to its base point \( m\).
Let \( \Hol_g (\gamma) \) be the holonomy of \( \gamma \) relative to the 
connection on the canonical bundle \( \KBundle_g M \) induced by the Levi--Civita 
connection of \( g \). The Stokes  theorem implies
\begin{equation}
\SectionMapAbb{J}(g) (\gamma) = \Hol_g (\gamma) = \exp \int_{[0, 1] \times [0, 1]} 
\Sigma^* \, (- S_g \mu) \equiv \exp \Bigl( - \int_\Sigma S_g \mu \, \Bigr),
\end{equation}
because the connection on \( \KBundle_g M \) has curvature \( - S_g \mu \).
According to \parencite[Lemma~2.4.1]{MarsdenEbinEtAl1972}, the derivative of the 
map \( S: g \mapsto S_g \) in the direction \( h \) 
is given by
\begin{equation}
	\label{eq:derivScalarCurvPrelim}
	\tangent_g S (h) 
		= \laplace (\tensor{h}{^i_i}) + \nabla_i \nabla_j h^{ij} - R_{ij} h^{ij}.
\end{equation}
As \( M \) is \( 2 \)-dimensional, the Ricci curvature \( R_{ij} \) of \( g \) 
satisfies \( R_{ij} = \frac{S_g}{2} g_{ij} \).
Thus, the last term in~\eqref{eq:derivScalarCurvPrelim} is proportional to the 
trace of \( h \).
However, if \( h \) is a tangent vector at \( g \) to \( \MetricSpace_\mu \), 
then its trace vanishes and we then get
\begin{equation}
	\tangent_g S (h) 
		= \nabla_i \nabla_j h^{ij}.
\end{equation}
Hence, for the logarithmic derivative of the holonomy of \( \gamma \), we obtain
\begin{equation}
	\bigl(\difLog \Hol (\gamma)\bigr)_g (h)
		= - \int_\Sigma \nabla_i \nabla_j h^{ij} \, \mu
\end{equation}
for \( g \in \MetricSpace_\mu \) and \( h \in \TBundle_g \MetricSpace_\mu \).

Since \( \deRCohomology^1(M, \R) = 0 \), the exterior differential 
\( \dif: \DiffFormSpace^1(M) \to \DiffFormSpace^2(M) \) yields an isomorphism 
of \( \DiffFormSpace^1(M) \slash \dif \DiffFormSpace^0(M) \) with 
\( \dif \DiffFormSpace^1(M) \subseteq \DiffFormSpace^2(M) \).
It suffices to show that the \( 1 \)-form \( \alpha_i = 
\tensor{\mu}{_i_k} \nabla_l \tensor{h}{^k^l} \) satisfies \( \dif \alpha = 
- \nabla_k \nabla_l h^{kl} \, \mu \), because then
\begin{equation}
	\int_\gamma (\difLog \SectionMapAbb{J})_g (h)
		= \bigl(\difLog \Hol (\gamma)\bigr)_g (h)
		= \int_\Sigma \dif \alpha 
		= \int_\gamma \alpha
\end{equation}
and so \( (\difLog \SectionMapAbb{J})_g(h) = \alpha\mod\dif \DiffFormSpace^0(M) \) 
as \( \gamma \) was an arbitrary closed loop.
For this, note that every vector field \( Y^k \) satisfies
\begin{equation}
	\nabla_i (Y^k \mu_{kj}) - \nabla_j (Y^k \mu_{ki})
		= \bigl(\dif (Y \contr \mu)\bigr)_{ij}
		= (\difLie_Y \mu)_{ij}
		= (\divergence Y) \, \mu_{ij}
		= (\nabla_k Y^k) \, \mu_{ij} \, .
\end{equation}
Applying this identity with \( Y^k = \nabla_l \tensor{h}{^k^l} \), we find
\begin{equation}\begin{split}
	(\dif \alpha)_{ij}
		&= \nabla_i \alpha_j - \nabla_j \alpha_i \\
		&= \nabla_i (\tensor{\mu}{_j_k} \nabla_l \tensor{h}{^k^l}) - 
		\nabla_j (\tensor{\mu}{_i_k} \nabla_l \tensor{h}{^k^l}) \\
		&= - (\nabla_k \nabla_l \tensor{h}{^k^l}) \, \mu_{ij}
\end{split}\end{equation}
and the claim follows.
\end{proof}

\begin{refcontext}[sorting=nyt]{}
	\printbibliography
\end{refcontext}

\medskip

\begin{flushleft}
Tobias Diez \\
Max Planck Institute for Mathematics in the Sciences, \\
04103 Leipzig, Germany, \\
Institute of Applied Mathematics, \\
Delft University of Technology, \\
2628 XE Delft, Netherlands \\
E-mail address: \texttt{T.Diez@tudelft.nl}

Supported by the NWO grant 639.032.734.
\end{flushleft}
\begin{flushleft}
Tudor S. Ratiu \\
School of Mathematical Sciences, \\
Shanghai Jiao Tong University, Minhang District, \\
800 Dongchuan Road, 200240 China, \\
Section de Math\'ematiques, Universit\'e de Gen\`eve, \\
1211 Gen\`eve 4, \\
Ecole Polytechnique F\'ed\'eralede Lausanne, \\ 
1015 Lausanne,  Switzerland \\
E-mail address: \texttt{ratiu@sjtu.edu.cn, tudor.ratiu@epfl.ch}

Partially supported by National Natural Science Foundation of China grant 11871334 and NCCR SwissMAP grant of the Swiss National Science Foundation.
\end{flushleft}

\end{document}